\documentclass[12pt]{article}
\usepackage[english]{babel}
\usepackage[all]{xy}
\usepackage{fancyhdr}
\usepackage{setspace, amsmath, amsthm, amssymb, amsfonts, amscd, epic, graphicx, ulem, dsfont}
\usepackage{amsthm}
\usepackage[T1]{fontenc}
\usepackage{cases}
\newtheorem{theorem}{Theorem}

\newtheorem{lemma}[theorem]{Lemma}
\newtheorem{proposition}[theorem]{Proposition}

\newtheorem{example}[theorem]{Example}

\newtheorem{remark}[theorem]{Remark}

\makeatletter

\@addtoreset{equation}{section}
\makeatother

\title{Bi-symphonic maps between Riemannian manifolds}

\author{Ahmed Mohammed Cherif\footnote{University Mustapha Stambouli Mascara, Faculty of Exact Sciences, Mascara 29000, Algeria. Email: a.mohammedcherif@univ-mascara.dz}
and Kaddour Zegga\footnote{University Mustapha Stambouli Mascara, Faculty of Exact Sciences, Mascara 29000, Algeria. Email: zegga.kadour@univ-mascara.dz}}
\date{}

\begin{document}
\maketitle
	
\begin{abstract}
This note introduces an extension to the definition of symphonic maps, denoted as $\varphi:(M,g)\longrightarrow(N,h)$, by exploring variations in the bi-energy functional associated with the pullback metric $\varphi^*h$ between two Riemannian manifolds.\\
\textit{Keywords:}  Symphonic map, Variational problem, Bi-Symphonic map.\\
\textit{Mathematics Subject Classification 2020:} 58E20,  53C43.
\end{abstract}

\section{Introduction}

The term << symphonic maps >> in the context of mathematics, particularly in differential geometry, refers to a specialized concept related to the study of energy functionals and critical points.  The concept of symphonic maps in mathematics is associated with the $m$-symphonic energy, a new energy functional of conformal invariance, and the study of critical points of this energy functional. The term << symphonic maps >> has been used in the context of constructing symphonic maps between ellipsoids \cite{Cao}, and studying the regularity of critical points of the $m$-symphonic energy, particularly in higher dimensions $(m \geq 4)$, \cite{MISN1}.
We consider a functional of pullbacks of metrics on the space of maps $\varphi$ between Riemannian manifolds. Harmonic maps are stationary points of the energy functional $E(\varphi)$ which is an integral of the trace of the pullback of the metric of the target manifold by $\varphi$.
Our functional $E_{sym}(\varphi)$
 is an integral of the norm of the pullback. Stationary maps for $E_{sym}(\varphi)$
 are called as symphonic maps \cite{KN,KN2,MN,MISN,MISN1,NT,NTW}.\\
In this paper, we extend the definition of symphonic maps $\varphi$ from $(M,g)$ to $(N,h)$ via the variation of the bi-energy functional related to the pullback metric $\varphi^*h$ between two Riemannian manifolds.\\
Let $(M^{m}, g)$, $(N^{n}, h)$ be Riemannian manifolds without boundary, and let $\varphi$ be a smooth map from $M$ into $N$.  Let $\varphi^{*}h$ be the
pullback of the metric $h$ by $\varphi$, i.e.,
$$( \varphi^{*}h)(X, Y ) = h(d \varphi (X), d \varphi (Y )),$$
for any vector fields $X$, $Y$ on $M$. We consider the functional or the symphonic energy
$$E_{sym}(\varphi)=\int_{M}\|\varphi^{*}h\|^{2}dv_{g},$$
where $dv_{g}$ is the volume form on $(M, g)$, and  $\|\varphi^{*}h\|$ denotes the norm of the pullback  $\varphi^{*}h$, i.e.,
 $$\|\varphi^{*}h\|^{2}=\sum_{i,j=1}^{m}h(d\varphi(e_{i}),d\varphi(e_{j}))^{2},$$  and
 $\{e_{i}\}_{i=1,...,m}$ is a local orthonormal frame on $(M^{m}, g)$.
 The symphonic energy $E_{sym}(\varphi)$ is related to the energy $E(\varphi)$ in the theory of harmonic maps since the functional $E_{sym}(\varphi)$ is an integral of the norm of the pullback  $\varphi^{*}h$, while the energy $E(\varphi)$ is an integral of the trace of the pullback $\varphi^{*}h$. Indeed, the energy $E(\varphi )$ is defined to be
$$ E(\varphi)=\int_{M}\|d\varphi\|^{2}dv_{g},$$
where
$\|d\varphi\|^{2} =\displaystyle\sum_{i=1}^{m}h(d\varphi(e_{i}), d\varphi(e_{i}))$.  A map $\varphi$ is called harmonic if it is a critical point of the functional energy $E(\varphi)$, i.e., if the first variation of $E(\varphi)$ at $\varphi$ vanishes.
If $M$ is noncompact, $\varphi$ is defined to be a harmonic map if it is a critical point of the functional energy $E(\varphi)$, on any compact subdomain of $M$. The theory of harmonic maps is developed with applications to other fields \cite{EL}, \cite{ES}. The first variation formula of the  functional energy  $E_{sym}(\varphi)$  was given by Shigeo Kawaia, Nobumitsu Nakauchi in \cite{KN}.
\begin{proposition}[First Variation Formula \cite{KN}]
We have
$$\frac{dE_{sym}(\varphi_{t})}{dt}\Big|_{t=0}=-4\int_{M}h(\operatorname{div}_{g}\sigma_{\varphi},\upsilon)dv_{g},$$
for any variation vector field $\upsilon$, where  $\sigma_{\varphi}$ is defined by
$$\sigma_{\varphi}(X)=\sum_{j=1}^{m}h(d\varphi(X),d\varphi(e_{j}))d\varphi(e_{j}), \quad  \forall  X \in \Gamma(TM), $$ and  the divergence of $\sigma_{\varphi}$ is given  by
\begin{eqnarray}\label{EQ1}
\nonumber \operatorname{div}_{g}\sigma_{\varphi}
&=&\sum_{i,j=1}^{m}\Big\{ h(\nabla d\varphi(e_{i},e_{i}),d\varphi(e_{j}))d\varphi(e_{j}) + h(d\varphi(e_{i}),\nabla d\varphi(e_{i},e_{j}))d\varphi(e_{j})\\
&+& h(d\varphi(e_{i}),d\varphi(e_{j}))\nabla d\varphi(e_{i},e_{j})\Big\},
\end{eqnarray}
here $\nabla d\varphi$ denotes the second fundamental form of $\varphi$.
\end{proposition}
Recall that the map $\varphi$ is called symphonic if it is a critical point of the functional energy
 $E_{sym}$.
We denote by  $\tau^{s}(\varphi)=\operatorname{div}_{g}\sigma_{\varphi}$. Thus,
the map $\varphi$ is called symphonic if and only if  $\tau^{s}(\varphi)=0$.
\begin{example}
Let $(M,g)$ be a Riemannian manifold.
A smooth function $f:(M,g)\longrightarrow\mathbb{R}$ is symphonic if and only if
$$(\Delta f)\|\operatorname{grad}f\|^2+2\operatorname{Hess}_f(\operatorname{grad}f,\operatorname{grad}f)=0,$$
where $\operatorname{grad} f$ (resp. $\operatorname{Hess}_f$) is the gradient vector (resp.  the Hessian) of $f$ with respect to $g$.
\begin{itemize}
  \item If $(M,g)=(\mathbb{R}^n,dx_1^2+...+dx_n^2)$, the last equation is given by
$$\sum_{i,j=1}^n\left[\frac{\partial^2 f}{\partial x_i^2}\left(\frac{\partial f}{\partial x_j}\right)^2+2\frac{\partial f}{\partial x_i}\frac{\partial f}{\partial x_j}\frac{\partial^2 f}{\partial x_i\partial x_j}\right]=0.$$
  \item Let $(M,g)=(\mathbb{R}^2\backslash\{0\},dx_1^2+dx_2^2)$, then the function $f(x_1,x_2)=(x_1^2+x_2^2)^{\frac{1}{3}}$ is symphonic non-harmonic on $(M,g)$.
\end{itemize}
\end{example}
\section{Main results}
\subsection{Second variation formula}
Take any smooth deformation $\phi$ of $\varphi$ with two parameters, i.e., any smooth map
$\phi:(-\epsilon,\epsilon)\times (-\delta,\delta) \times M\longrightarrow N$  such that
$\phi(0,0,x)=\varphi(x)$.\\
Let $\varphi_{s,t}(x)=\phi(s,t,x)$, and we often say a deformation   $\varphi_{s,t}(x)$  instead of a deformation   $\phi(s, t, x)$. Let $\upsilon=d\phi(\frac{\partial}{\partial t})|_{s=t=0}$ and $w=d\phi(\frac{\partial}{\partial s})|_{s=t=0}$,  denote the variation vector fields of the deformation
$\varphi_{s,t}$.\\
Let $\{e_{i}\}_{i=1,...,m}$  be an orthonormal frame with respect to $g$ on $M$, such that $\nabla^{M}_{e_{i}}e_{j}=0$, at fixed point $x \in M$ for all $i, j = 1, ...,m$.\\
 In \cite{KN}, the authors calculated the second variation formula which was used in the study of stability of stationary maps. In the following Theorem we compute the second variation formula where we deduce the symphonic-Jacobi operator which shows the relationship between the first variation of the bi-energy functional in the next section.
\begin{theorem}\label{thm1}
Let $\varphi : (M^{m}, g)\longrightarrow (N^{n}, h)$ be a symphonic map between Riemannian manifolds. Under the notation above we have the following
$$\frac{d^{2}}{ds dt}E_{sym}(\varphi_{s,t})\Big|_{s=t=0}=-4\int_{M}h(J^{s}_{\varphi}(\upsilon),w)dv_{g} ,$$
 where $J^{s}_{\varphi}$ is the symphonic-Jacobi operator corresponding to $\varphi$ given by
 \begin{eqnarray*}
   J^{s}_{\varphi}(\upsilon) &=& \sum_{i,j=1}^{m}\Big\{2h(\nabla^{\varphi}_{e_{i}}\upsilon,d\varphi(e_{j}))\nabla d\varphi(e_{i},e_{j}) \\
   &+&\Big[ h(\nabla^{\varphi}_{e_{i}}\nabla^{\varphi}_{e_{i}}\upsilon-
   \nabla^{\varphi}_{\nabla^{M}_{e_{i}}e_{i}}\upsilon,d\varphi(e_{j}))+
   h(\nabla^{\varphi}_{e_{j}}\upsilon,\nabla d\varphi(e_{i},e_{i}))\Big]d\varphi(e_{j})\\ &+&
    \Big[ h(\nabla d\varphi(e_{i},e_{j}),d\varphi(e_{j}))+
   h(d\varphi(e_{i}),\nabla d\varphi(e_{j},e_{j}))\Big]\nabla^{\varphi}_{e_{i}}\upsilon \\
   &+& h(d\varphi(e_{i}),d\varphi(e_{j}))\Big[\nabla^{\varphi}_{e_{j}}\nabla^{\varphi}_{e_{i}}\upsilon-
   \nabla^{\varphi}_{\nabla^{M}_{e_{j}}e_{i}}\upsilon +
   R^{N}(\upsilon,d\varphi(e_{j}))d\varphi(e_{i})\Big]\Big\},
 \end{eqnarray*}
and $\nabla^{\varphi}$ denotes the pull-back connection on $\varphi^{-1}TN$.
\end{theorem}
\begin{proof}[Proof of Theorem \ref{thm1}]
In the following, we note that  $\partial t =\frac{\partial}{\partial t}$   and  $\partial s =\frac{\partial}{\partial s}$, we compute
\begin{eqnarray}\label{EQ2}
\nonumber \frac{\partial^{2}}{\partial s \partial t}E_{sym}(\varphi_{s,t})\Big|_{s=t=0}
&=& \frac{\partial}{\partial s}\Big[\frac{\partial}{\partial t}
   \int_{M}\|\varphi^{*}_{s,t}h\|^{2}dv_{g}\Big]\Big|_{s=t=0}\\ \nonumber
   &=& 4 \frac{\partial}{\partial s}\Big[\sum_{i,j=1}^{m}\int_{M}h(d\varphi_{s,t}(e_{i}),d\varphi_{s,t}(e_{j}))
   h(\nabla^{\phi}_{e_{i}}d\phi(\partial t),d\varphi_{s,t}(e_{j}))dv_{g} \Big]\Big|_{s=t=0} \\ \nonumber
   &=& 4\sum_{i,j=1}^{m}\Big\{\int_{M}h(\nabla^{\phi}_{\partial s}d\varphi_{s,t}(e_{i}),d\varphi_{s,t}(e_{j}))h(\nabla^{\phi}_{e_{i}}d\phi(\partial t),d\varphi_{s,t}(e_{j}))\Big|_{s=t=0}dv_{g}  \\ \nonumber
   &+&  \int_{M}h(d\varphi_{s,t}(e_{i}),\nabla^{\phi}_{\partial s}d\varphi_{s,t}(e_{j}))h(\nabla^{\phi}_{e_{i}}d\phi(\partial t),d\varphi_{s,t}(e_{j}))\Big|_{s=t=0}dv_{g} \\ \nonumber
   &+&  \int_{M}h(d\varphi_{s,t}(e_{i}),d\varphi_{s,t}(e_{j}))h(\nabla^{\phi}_{\partial s}\nabla^{\phi}_{e_{i}}d\phi(\partial t),d\varphi_{s,t}(e_{j}))\Big|_{s=t=0}dv_{g} \\
   &+& \int_{M}h(d\varphi_{s,t}(e_{i}),d\varphi_{s,t}(e_{j}))h(\nabla^{\phi}_{e_{i}}d\phi(\partial t),\nabla^{\phi}_{\partial s}d\varphi_{s,t}(e_{j}))\Big|_{s=t=0}dv_{g}\Big\}.
\end{eqnarray}
We denote by $T_1$ the first term in the right hand of (\ref{EQ2}), we have
\begin{eqnarray}\label{T1}
\nonumber \frac{1}{4}T_{1} &=&\sum_{i,j=1}^{m}\int_{M}h(\nabla^{\phi}_{e_{i} }d\phi(\partial s),d\varphi_{s,t}(e_{j}))h(\nabla^{\phi}_{e_{i}}d\phi(\partial t),d\varphi_{s,t}(e_{j}))\Big|_{s=t=0}dv_{g} \\&=&\nonumber\sum_{i,j=1}^{m}\int_{M}e_{i}\Big[h(d\phi(\partial s),d\varphi_{s,t}(e_{j}))h(\nabla^{\phi}_{e_{i}}d\phi(\partial t),d\varphi_{s,t}(e_{j}))\Big]\Big|_{s=t=0}dv_{g} \\
&-&\nonumber\sum_{i,j=1}^{m}\int_{M}h(d\phi(\partial s),\nabla^{\phi}_{e_{i}}d\varphi_{s,t}(e_{j}))h(\nabla^{\phi}_{e_{i}}d\phi(\partial t),d\varphi_{s,t}(e_{j}))\Big|_{s=t=0}dv_{g}  \\
&-&\nonumber\sum_{i,j=1}^{m}\int_{M}h(d\phi(\partial s),d\varphi_{s,t}(e_{j}))h(\nabla^{\phi}_{e_{i}}\nabla^{\phi}_{e_{i}}d\phi(\partial t),d\varphi_{s,t}(e_{j}))\Big|_{s=t=0}dv_{g} \\
&-&\sum_{i,j=1}^{m}\int_{M}h(d\phi(\partial s),d\varphi_{s,t}(e_{j}))h(\nabla^{\phi}_{e_{i}}d\phi(\partial t),\nabla^{\phi}_{e_{i}}d\varphi_{s,t}(e_{j})) \Big|_{s=t=0}dv_{g}.\qquad\qquad
  \end{eqnarray}
  We consider the differential $1$-form on $M$, defined by
  $$\omega_{1}(X)=\displaystyle \sum_{j=1}^{m}h(d\phi(\partial s),d\varphi_{s,t}(e_{j}))h(\nabla^{\phi}_{X}d\phi(\partial t),d\varphi_{s,t}(e_{j})),$$
   then, the divergence of $\omega_1$ at $x$ is given by
  $$\mathop{\operatorname{div}}(\omega_{1})=\sum_{i,j=1}^{m}e_{i}\Big[h(d\phi(\partial s),d\varphi_{s,t}(e_{j}))h(\nabla^{\phi}_{e_{i}}d\phi(\partial t),d\varphi_{s,t}(e_{j}))\Big].$$
Now, by using the divergence Theorem (see \cite{BW}), we get
\begin{eqnarray}\label{ET1}
\nonumber\frac{1}{4}T_{1} &=&\nonumber-\sum_{i,j=1}^{m}\int_{M}h\Big(w,h(\nabla^{\varphi}_{e_{i}}\upsilon ,d\varphi(e_{j}))\nabla d\varphi(e_{i},e_{j})\\\nonumber &+&h(\nabla^{\varphi}_{e_{i}}\nabla^{\varphi}_{e_{i}}\upsilon ,d\varphi(e_{j}))d\varphi(e_{j})\\&+& h(\nabla^{\varphi}_{e_{i}}\upsilon ,\nabla d\varphi(e_{i},e_{j}))d\varphi(e_{j})\Big)dv_{g}.
\end{eqnarray}
For the second term  in the right hand of  (\ref{EQ2}), we have
\begin{eqnarray}\label{T2}
  \nonumber\frac{1}{4}T_{2}&=&\sum_{i,j=1}^{m}\int_{M}h(d\varphi(e_{i}),\nabla^{\varphi}_{e_{j}}w)
  h(\nabla^{\varphi}_{e_{i}}\upsilon,d\varphi(e_{j}))dv_{g}  \\
 \nonumber&=&\sum_{i,j=1}^{m}\Big\{ \int_{M}e_{j}[h(d\varphi(e_{i}),w)h(\nabla^{\varphi}_{e_{i}}\upsilon ,d\varphi(e_{j}))]dv_{g}\\
\nonumber &-& \int_{M}h(\nabla^{\varphi}_{e_{j}}d\varphi(e_{i}),w)h(\nabla^{\varphi}_{e_{i}}\upsilon,d\varphi(e_{j}))dv_{g}\\
\nonumber &-& \int_{M}h(d\varphi(e_{i}),w)h(\nabla^{\varphi}_{e_{j}}\nabla^{\varphi}_{e_{i}}\upsilon,d\varphi(e_{j}))dv_{g}\\
&-& \int_{M}h(d\varphi(e_{i}),w)h(\nabla^{\varphi}_{e_{i}}\upsilon,\nabla^{\varphi}_{e_{j}}d\varphi(e_{j}))dv_{g}\Big\}.
\end{eqnarray}
 By using the divergence Theorem, we obtain
\begin{eqnarray}\label{ET2}
\nonumber \frac{1}{4}T_{2} &=& - \int_{M}h\Big(w,\sum_{i,j=1}^{m}h(\nabla^{\varphi}_{e_{i}}\upsilon ,d\varphi(e_{j}))\nabla d\varphi(e_{i},e_{j})\\ \nonumber &+& \sum_{i,j=1}^{m}h(\nabla^{\varphi}_{e_{j}}\nabla^{\varphi}_{e_{i}}\upsilon ,d\varphi(e_{j}))d\varphi(e_{i})\\ &+& \sum_{i=1}^{m}h(\nabla^{\varphi}_{e_{i}}\upsilon ,\tau(\varphi))d\varphi(e_{i})\Big)dv_{g},
\end{eqnarray}
where $\tau(\varphi)$ is the tension field of $\varphi$ (see \cite{BW,ES}).
For the third  term  in the right hand of (\ref{EQ2}), we have
\begin{eqnarray*}
   \frac{1}{4}T_{3}&=& \sum_{i,j=1}^{m}\int_{M}h(d\varphi_{s,t}(e_{i}),d\varphi_{s,t}(e_{j}))h(\nabla^{\phi}_{\partial s}\nabla^{\phi}_{e_{i}}d\phi(\partial t),d\varphi_{s,t}(e_{j}))\Big|_{s=t=0}dv_{g} \\
   &=& \sum_{i,j=1}^{m}\int_{M}h(d\varphi_{s,t}(e_{i}),d\varphi_{s,t}(e_{j}))h(R^{N}(d\varphi(\partial s),d\varphi(e_{i}))d\phi(\partial t),d\varphi_{s,t}(e_{j}))\Big|_{s=t=0} dv_{g} \\
   &+&  \sum_{i,j=1}^{m}\int_{M}h(d\varphi_{s,t}(e_{i}),d\varphi_{s,t}(e_{j}))
   h(\nabla^{\phi}_{e_{i}}\nabla^{\phi}_{\partial s}d\phi(\partial t),d\varphi_{s,t}(e_{j}))\Big|_{s=t=0}dv_{g}\\
   &=&-\sum_{i,j=1}^{m}\int_{M}h(d\varphi_{s,t}(e_{i}),d\varphi_{s,t}(e_{j}))h(R^{N}(d\phi(\partial t),d\varphi_{s,t}(e_{j}))d\varphi_{s,t}(e_{i}),d\phi(\partial s))\Big|_{s=t=0}dv_{g}\\
   &+& \sum_{i,j=1}^{m}\int_{M}e_{i}\Big(h(d\varphi_{s,t}(e_{i}),d\varphi_{s,t}(e_{j}))h(\nabla^{\phi}_{\partial s}d\phi(\partial t),d\varphi_{s,t}(e_{j}))\Big)\Big|_{s=t=0}dv_{g}\\
   &-& \displaystyle\int_{M}h\Big(\nabla^{\phi}_{\partial s}d\phi(\partial t)\Big|_{s=t=0},\tau^{s}(\varphi)\Big)dv_{g}.
\end{eqnarray*}
  Using the symphonic condition of $\varphi$ and the divergence Theorem, we deduce
\begin{eqnarray}\label{ET3}
   \frac{1}{4}T_{3}&=& -\sum_{i,j=1}^{m}\int_{M}h\Big(w,h(d\varphi(e_{i}),d\varphi(e_{j}))R^{N}(\upsilon,d\varphi(e_{j}))d\varphi(e_{i})    \Big)dv_{g},\qquad
\end{eqnarray}
of the same method, we deduct the last term of (\ref{EQ2})
\begin{eqnarray}\label{ET4}
\nonumber \frac{1}{4}T_{4}&=&-\int_{M}h\Big(w,\sum_{i,j=1}^{m}h(\nabla d\varphi(e_{i},e_{j}),d\varphi(e_{j}))\nabla^{\varphi}_{e_{i}}\upsilon\\ &+& \nonumber \sum_{i=1}^{m} h(d\varphi(e_{i}),\tau(\varphi))\nabla^{\varphi}_{e_{i}}\upsilon \\
  &+&\sum_{i,j=1}^{m}h(d\varphi(e_{i}),d\varphi(e_{j}))\nabla^{\varphi}_{e_{j}}\nabla^{\varphi}_{e_{i}}\upsilon \Big)dv_{g}.
 \end{eqnarray}
By substituting (\ref{ET1}), (\ref{ET2}), (\ref{ET3}) and (\ref{ET4}) in (\ref{EQ2}), we get the result of Theorem \ref{thm1}.
\end{proof}

\subsection{The first variation of the bi-energy functional}

Let $\varphi:(M,g)\longrightarrow(N,h)$ be a smooth map between two Riemannian manifolds.
We consider the bi-energy functional related to the pullback $\varphi^{*}h$ by
$$E_{2,sym}(\varphi)=\int_{M}\|\tau^{s}(\varphi)\|^{2}dv_{g}.$$
The map $\varphi$ is called bi-symphonic if it is a critical  point of the bi-energy functional
$E_{2,sym}(\varphi)$,  i.e.,  $$\frac{d}{dt}E_{2,sym}(\varphi_{t})\Big|_{t=0}=0,$$
for any smooth deformation  $\phi$ of $\varphi$  with one parameter, i.e., any smooth map
$\phi:(-\epsilon,\epsilon) \times M\longrightarrow N$, $(t,x)\longmapsto \varphi_t(x)$  such that $\phi(0,x)=\varphi(x)$.
Let  $\upsilon =d\phi(\frac{\partial}{\partial t})\Big|_{t=0}$  denote the variation vector field of the deformation  $\phi$.  Let $\{e_{i}\}_{i=1,...,m}$  be an orthonormal frame with respect to $g$ on $M$, such that  $\nabla^{M}_{e_{i}}e_{j}=0$, at fixed point $x \in M$ for all $i, j = 1, ...,m$. Under the notation above we have the following.
\begin{theorem}\label{thm2}
Let $\varphi : (M^{m}, g)\longrightarrow (N^{n}, h)$ be a smooth map between Riemannian manifolds. Then
$$\frac{d}{dt}E_{2,sym}(\varphi_{t})\Big|_{t=0}=-\int_{M}h( \upsilon ,\tau^{s}_{2}(\varphi))dv_{g},$$
where $\tau^{s}_{2}(\varphi)$ is given by
\begin{eqnarray*}
  \tau^{s}_{2}(\varphi) &=& \sum_{i,j=1}^{m}\Big\{2h(\nabla^{\varphi}_{e_{i}}\tau^{s}(\varphi),d\varphi(e_{j}))\nabla d\varphi(e_{i},e_{j}) \\
   &+&\Big[ h(\nabla^{\varphi}_{e_{i}}\nabla^{\varphi}_{e_{i}}\tau^{s}(\varphi)-
   \nabla^{\varphi}_{\nabla^{M}_{e_{i}}e_{i}}\tau^{s}(\varphi),d\varphi(e_{j}))+
   h(\nabla^{\varphi}_{e_{j}}\tau^{s}(\varphi),\nabla d\varphi(e_{i},e_{i}))\Big]d\varphi(e_{j})\\ &+&
    \Big[ h(\nabla d\varphi(e_{i},e_{j}),d\varphi(e_{j}))+
   h(d\varphi(e_{i}),\nabla d\varphi(e_{j},e_{j}))\Big]\nabla^{\varphi}_{e_{i}}\tau^{s}(\varphi) \\
   &+& h(d\varphi(e_{i}),d\varphi(e_{j}))\Big[\nabla^{\varphi}_{e_{j}}\nabla^{\varphi}_{e_{i}}\tau^{s}(\varphi)-
   \nabla^{\varphi}_{\nabla^{M}_{e_{j}}e_{i}}\tau^{s}(\varphi) +
   R^{N}(\tau^{s}(\varphi),d\varphi(e_{j}))d\varphi(e_{i})\Big]\Big\}.
\end{eqnarray*}
\end{theorem}
For the proof of Theorem \ref{thm2}, first we have
$$\frac{d}{dt}E_{2,sym}(\varphi_{t})\Big|_{t=0}=\int_{M}\frac{\partial}{\partial t}h(\tau^{s}(\varphi_{t}),\tau^{s}(\varphi_{t}))\Big|_{t=0}dv_{g}= 2\int_{M}h(\frac{\partial}{\partial t}\tau^{s}(\varphi_{t}),\tau^{s}(\varphi_{t}))\Big|_{t=0}dv_{g},$$
we compute the term $\frac{\partial}{\partial t}\tau^{s}(\varphi_{t})$. We recall that
\begin{eqnarray*}
\tau^{s}(\varphi)
&=&\sum_{i,j=1}^{m}\{ h(\tau(\varphi),d\varphi(e_{j}))d\varphi(e_{j}) + h(d\varphi(e_{i}),\nabla d\varphi(e_{i},e_{j}))d\varphi(e_{j})\\
&&+ h(d\varphi(e_{i}),d\varphi(e_{j}))\nabla d\varphi(e_{i},e_{j})\}.
\end{eqnarray*}
We obtain
\begin{eqnarray}\label{E11}
 \nonumber \frac{\partial}{\partial t}\tau^{s}(\varphi_{t})&=& \sum_{i,j=1}^{m}\{ \frac{\partial}{\partial t} [h(\tau(\varphi_{t}),d\varphi_{t}(e_{j}))d\varphi_{t}(e_{j})] +\frac{\partial}{\partial t}[ h(d\varphi_{t}(e_{i}),\nabla d\varphi_{t}(e_{i},e_{j}))d\varphi_{t}(e_{j})]\\
&+& \frac{\partial}{\partial t}[h(d\varphi_{t}(e_{i}),d\varphi_{t}(e_{j}))\nabla d\varphi_{t}(e_{i},e_{j})]\}.
\end{eqnarray}
So, we need the following Lemmas.
\begin{lemma}\label{lem1}
 The first term of (\ref{E11}) is given by
\begin{eqnarray*}
\sum_{i,j=1}^{m} h\Big( \frac{\partial}{\partial t} h(\tau(\varphi_{t}),d\varphi_{t}(e_{j}))d\varphi_{t}(e_{j}),\tau^{s}(\varphi_{t})\Big)\Big|_{t=0}
&=&\sum_{i,j=1}^{m}\Big\{-
h(\nabla^{\varphi}_{e_{i}}\upsilon,\nabla^{\varphi}_{e_{i}}d\varphi(e_{j}))h(d\varphi(e_{j}),\tau^{s}(\varphi))\\
&+& h(R^{N}(\upsilon,d\varphi(e_{i}))d\varphi(e_{i}),d\varphi(e_{j}))h(d\varphi(e_{j}),\tau^{s}(\varphi))\\&-&
h(\nabla^{\varphi}_{e_{i}}\upsilon,d\varphi(e_{j}))h(\nabla^{\varphi}_{e_{i}}d\varphi(e_{j}),\tau^{s}(\varphi))\\&-&
h(\nabla^{\varphi}_{e_{i}}\upsilon,d\varphi(e_{j}))h(d\varphi(e_{j}),\nabla^{\varphi}_{e_{i}}\tau^{s}(\varphi))\Big\}
\\&+& \sum_{j=1}^{m}\Big\{h(\tau(\varphi),\nabla^{\varphi}_{e_{j}}\upsilon)h(d\varphi(e_{j}),\tau^{s}(\varphi))\\
&+&  h(\tau(\varphi),d\varphi(e_{j}))h(\nabla^{\varphi}_{e_{j}}\upsilon,\tau^{s}(\varphi))\Big\}+\operatorname{div}_g\eta_{1},
\end{eqnarray*}
where $\eta_{1}(X)=\displaystyle \sum_{j=1}^{m}h(\nabla^{\varphi}_{X}\upsilon,d\varphi(e_{j}))h(d\varphi(e_{j}),\tau^{s}(\varphi))$.
\end{lemma}
\begin{proof}[Proof of Lemma \ref{lem1}]
 We have
\begin{eqnarray*}
 \sum_{j=1}^{m} \frac{\partial}{\partial t} [h(\tau(\varphi_{t}),d\varphi_{t}(e_{j}))d\varphi_{t}(e_{j})] &=& \sum_{j=1}^{m}\Big\{ h(\nabla^{\phi}_{\partial t}\tau(\varphi_{t}),d\varphi_{t}(e_{j}))d\varphi_{t}(e_{j}) \\ &+& h(\tau(\varphi_{t}),\nabla^{\phi}_{\partial t}d\varphi_{t}(e_{j}))d\varphi_{t}(e_{j})
+ h(\tau(\varphi_{t}),d\varphi_{t}(e_{j}))\nabla^{\phi}_{\partial t}d\varphi_{t}(e_{j})\Big\}\\
&=& \sum_{i,j=1}^{m}\Big\{ h(\nabla^{\phi}_{\partial t}\nabla^{\phi}_{e_{i}}d\varphi_{t}(e_{i}),d\varphi_{t}(e_{j}))d\varphi_{t}(e_{j}) \\ &+& h(\tau(\varphi_{t}),\nabla^{\phi}_{e_{j}}d\phi(\partial t))d\varphi_{t}(e_{j})
+  h(\tau(\varphi_{t}),d\varphi_{t}(e_{j}))\nabla^{\phi}_{e_{j}}d\phi(\partial t)\Big\}\\
&=&\sum_{i,j=1}^{m}\Big\{h\big(R^{N}(d\phi(\partial t),d\varphi_{t}(e_{i}))d\varphi_{t}(e_{i}),d\varphi_{t}(e_{j})\big)d\varphi_{t}(e_{j})\\ &+& h(\nabla^{\phi}_{e_{i}}\nabla^{\phi}_{e_{i}}d\phi(\partial t),d\varphi_{t}(e_{j}))d\varphi_{t}(e_{j})\\
&+&h(\tau(\varphi_{t}),\nabla^{\phi}_{e_{j}}d\phi(\partial t))d\varphi_{t}(e_{j})
+  h(\tau(\varphi_{t}),d\varphi_{t}(e_{j}))\nabla^{\phi}_{e_{j}}d\phi(\partial t)\Big\}.
\end{eqnarray*}
Now we calculate $\displaystyle \sum_{j=1}^{m} h\Big( \frac{\partial}{\partial t} [h(\tau(\varphi_{t}),d\varphi_{t}(e_{j}))d\varphi_{t}(e_{j})],\tau^{s}(\varphi_{t})\Big)\Big|_{t=0}$.
\begin{eqnarray}\label{E5}
\nonumber \sum_{j=1}^{m} h\Big( \frac{\partial}{\partial t} h(\tau(\varphi_{t}),d\varphi_{t}(e_{j}))d\varphi_{t}(e_{j}),\tau^{s}(\varphi_{t})\Big)\Big|_{t=0}
&=&\nonumber
\sum_{i,j=1}^{m} \Big\{h(\nabla^{\varphi}_{e_{i}}\nabla^{\varphi}_{e_{i}}\upsilon,d\varphi(e_{j}))h(d\varphi(e_{j}),\tau^{s}(\varphi))\\&+& \nonumber
h(R^{N}(\upsilon,d\varphi(e_{i}))d\varphi(e_{i}),d\varphi(e_{j}))h(d\varphi(e_{j}),\tau^{s}(\varphi))\Big\}
\\&+&\nonumber \sum_{j=1}^{m}\Big\{
h(\tau(\varphi),\nabla^{\varphi}_{e_{j}}\upsilon)h(d\varphi(e_{j}),\tau^{s}(\varphi))\\
&+&  h(\tau(\varphi),d\varphi(e_{j}))h(\nabla^{\varphi}_{e_{j}}\upsilon,\tau^{s}(\varphi))\Big\}.
\end{eqnarray}
 The first term of (\ref{E5}), becomes
\begin{eqnarray}\label{E6}
\nonumber \displaystyle\sum_{i,j=1}^{m}h(\nabla^{\varphi}_{e_{i}}\nabla^{\varphi}_{e_{i}}\upsilon,d\varphi(e_{j}))
h(d\varphi(e_{j}),\tau^{s}(\varphi))&=&\sum_{i,j=1}^{m}\Big\{
e_{i}\Big[h(\nabla^{\varphi}_{e_{i}}\upsilon,d\varphi(e_{j}))h(d\varphi(e_{j}),\tau^{s}(\varphi))\Big]\\&-&\nonumber
h(\nabla^{\varphi}_{e_{i}}\upsilon,\nabla^{\varphi}_{e_{i}}d\varphi(e_{j}))h(d\varphi(e_{j}),\tau^{s}(\varphi))\\&-&\nonumber
h(\nabla^{\varphi}_{e_{i}}\upsilon,d\varphi(e_{j}))h(\nabla^{\varphi}_{e_{i}}d\varphi(e_{j}),\tau^{s}(\varphi))\\&-&
h(\nabla^{\varphi}_{e_{i}}\upsilon,d\varphi(e_{j}))h(d\varphi(e_{j}),\nabla^{\varphi}_{e_{i}}\tau^{s}(\varphi))\Big\}.\qquad
\end{eqnarray}
Considering the differential $1$-form on $M$,  $$\eta_{1}(X)=\displaystyle \sum_{j=1}^{m}h(\nabla^{\varphi}_{X}\upsilon,d\varphi(e_{j}))h(d\varphi(e_{j}),\tau^{s}(\varphi)),$$ and replacing (\ref{E6}) in (\ref{E5}), we get the result of  Lemma \ref{lem1}.
\end{proof}
\begin{lemma}\label{Lem2}
 The second term of (\ref{E11}) is given by\\

 $\displaystyle\sum_{i,j=1}^{m}h\big(\frac{\partial}{\partial t}[ h(d\varphi_{t}(e_{i}),\nabla d\varphi_{t}(e_{i},e_{j}))d\varphi_{t}(e_{j})],\tau^{s}(\varphi_{t})\big)\Big|_{t=0}$
\begin{eqnarray*}
 \qquad\qquad\qquad\qquad&=&\sum_{i,j=1}^{m}\Big\{h\big(\nabla^{\varphi}_{e_{i}}\upsilon,\nabla d\varphi(e_{i},e_{j})\big)h(d\varphi(e_{j}),\tau^{s}(\varphi)) \\
 &-& h\Big(R^{N}(\upsilon,d\varphi(e_{i}))d\varphi(e_{i}), d\varphi(e_{j})\Big)h(d\varphi(e_{j}),\tau^{s}(\varphi))\\
 &-&h(\tau(\varphi),\nabla^{\varphi}_{e_{j}} \upsilon)h(d\varphi(e_{j}),\tau^{s}(\varphi))\\
 &-& h(d\varphi(e_{i}),\nabla^{\varphi}_{e_{j}} \upsilon)h(\nabla^{\varphi}_{e_{i}}d\varphi(e_{j}),\tau^{s}(\varphi))\\
 &+&h\Big(d\varphi(e_{i}),\nabla d\varphi(e_{i},e_{j})\Big)h(\nabla^{\varphi}_{e_{j}}\upsilon,\tau^{s}(\varphi))\\
 &-&h(d\varphi(e_{i}),\nabla^{\varphi}_{e_{j}}
 \upsilon)h(d\varphi(e_{j}),\nabla^{\varphi}_{e_{i}}\tau^{s}(\varphi))
 \Big\}+\operatorname{div}_g\eta_{2},
\end{eqnarray*}
where $\eta_{2}(X)=\displaystyle \sum_{j=1}^{m}h(d\varphi(X),\nabla^{\varphi}_{e_{j}} \upsilon)h(d\varphi(e_{j}),\tau^{s}(\varphi))$.
\end{lemma}

\begin{proof}[Proof of Lemma \ref{Lem2}] We compute
\begin{eqnarray*}
 \sum_{i,j=1}^{m} \frac{\partial}{\partial t}[ h(d\varphi_{t}(e_{i}),\nabla d\varphi_{t}(e_{i},e_{j}))d\varphi_{t}(e_{j})] &=& \sum_{i,j=1}^{m}\Big\{ h\Big(\nabla^{\phi}_{\partial t}d\varphi_{t}(e_{i}),\nabla d\varphi_{t}(e_{i},e_{j})\Big)d\varphi_{t}(e_{j})\\
  &+&  h\Big(d\varphi_{t}(e_{i}),\nabla^{\phi}_{\partial t}\nabla^{\phi}_{e_{i}} d\varphi_{}(e_{j})\Big)d\varphi_{t}(e_{j}) \\
   &+&  h\Big(d\varphi_{t}(e_{i}),\nabla d\varphi_{t}(e_{i},e_{j})\Big)\nabla^{\phi}_{\partial t}d\varphi_{t}(e_{j})\Big\}\\
    &=&\sum_{i,j=1}^{m}\Big\{h\Big(\nabla^{\phi}_{e_{i}}d\phi(\partial t),\nabla d\varphi_{t}(e_{i},e_{j})\Big)d\varphi_{t}(e_{j})\\
    &+& h\Big(d\varphi(e_{i}),R^{N}(d\phi(\partial t),d\varphi_{t}(e_{i})) d\varphi_{t}(e_{j})\Big)d\varphi_{t}(e_{j})\\
    &+& h\Big(d\varphi_{t}(e_{i}),\nabla^{\phi}_{e_{i}}\nabla^{\phi}_{e_{j}} d\phi(\partial t)\Big)d\varphi_{t}(e_{j})\\
    &+& h\Big(d\varphi_{t}(e_{i}),\nabla d\varphi_{t}(e_{i},e_{j})\Big)\nabla^{\phi}_{e_{i}}d\varphi_{t}(\partial t)\Big\}.
\end{eqnarray*}
So that\\

$\displaystyle\sum_{i,j=1}^{m}h(\frac{\partial}{\partial t}[ h(d\varphi(e_{i}),\nabla d\varphi(e_{i},e_{j}))d\varphi(e_{j})],\tau^{s}(\varphi))\Big|_{t=0}$
\begin{eqnarray}\label{E7}
\qquad\qquad\qquad\qquad&=&\nonumber\sum_{i,j=1}^{m}\Big\{h(\nabla^{\varphi}_{e_{i}}\upsilon,\nabla d\varphi(e_{i},e_{j}))h(d\varphi(e_{j}),\tau^{s}(\varphi)) \\
&+& \nonumber h(d\varphi(e_{i}),R^{N}(\upsilon,d\varphi_{t}(e_{i})) d\varphi(e_{j}))h(d\varphi(e_{j}),\tau^{s}(\varphi))\\&+& \nonumber h(d\varphi(e_{i}),\nabla^{\varphi}_{e_{i}}\nabla^{\varphi}_{e_{j}} \upsilon)h(d\varphi(e_{j}),\tau^{s}(\varphi))\\
&+& h(d\varphi(e_{i}),\nabla d\varphi(e_{i},e_{j}))h(\nabla^{\varphi}_{e_{i}}\upsilon,\tau^{s}(\varphi))\Big\}.
\end{eqnarray}
The third term of the right hand of (\ref{E7}), becomes
\begin{eqnarray}\label{E8}
\nonumber \displaystyle \sum_{i,j=1}^{m}h\left(d\varphi(e_{i}),\nabla^{\varphi}_{e_{i}}\nabla^{\varphi}_{e_{j}} \upsilon\right)h(d\varphi(e_{j}),\tau^{s}(\varphi)) &=& e_{i}\Big[\sum_{i,j=1}^{m}h(d\varphi(e_{i}),\nabla^{\varphi}_{e_{j}} \upsilon)h(d\varphi(e_{j}),\tau^{s}(\varphi))\Big] \\
\nonumber &-& \sum_{i,j=1}^{m}h(d\varphi(e_{i}),\nabla^{\varphi}_{e_{j}} \upsilon)h(\nabla^{\varphi}_{e_{i}}d\varphi(e_{j}),\tau^{s}(\varphi))\\ \nonumber
 &-&\sum_{i,j=1}^{m}h(d\varphi(e_{i}),\nabla^{\varphi}_{e_{j}}\upsilon)
 h(d\varphi(e_{j}),\nabla^{\varphi}_{e_{i}}\tau^{s}(\varphi))\\
 &-&\sum_{i,j=1}^{m}h(\tau(\varphi),\nabla^{\varphi}_{e_{j}} \upsilon)h(d\varphi(e_{j}),\tau^{s}(\varphi)).\qquad\qquad
\end{eqnarray}
Substituting (\ref{E8})  in (\ref{E7}),  we deduct the result of Lemma  (\ref{Lem2}).
\end{proof}
\begin{lemma}\label{Lem3}
 The third term of (\ref{E11}) is given by\\

 $\displaystyle\sum_{i,j=1}^{m} h\Big(\frac{\partial}{\partial t}[h(d\varphi_{t}(e_{i}),d\varphi_{t}(e_{j}))\nabla d\varphi_{t}(e_{i},e_{j})],\tau^{s}(\varphi_{t})\Big)\Big|_{t=0}$
\begin{eqnarray*}
\qquad\qquad\qquad\qquad&=&
\sum_{i,j=1}^{m}\Big\{h(\nabla^{\varphi}_{e_{i}}\upsilon,d\varphi(e_{j}))h(\nabla d\varphi(e_{i},e_{j}),\tau^{s}(\varphi))\\
 &+& h(d\varphi(e_{i}),\nabla^{\varphi}_{e_{j}}\upsilon)h(\nabla d\varphi(e_{i},e_{j}),\tau^{s}(\varphi))\\
   &+&h(d\varphi(e_{i}),d\varphi(e_{j}))h(R^{N}(\upsilon, d\varphi(e_{i}))d\varphi(e_{j}),\tau^{s}(\varphi))\\
   &-& h(\tau(\varphi),d\varphi(e_{j}))h(\nabla ^{\varphi}_{e_{j}}\upsilon,\tau^{s}(\varphi))\\
   &-&h(d\varphi(e_{i}),\nabla^{\varphi}_{e_{i}}d\varphi(e_{j}))h(\nabla^{\varphi}_{e_{j}}\upsilon,\tau^{s}(\varphi))\\
   &-&h(d\varphi(e_{i}),d\varphi(e_{j}))h(\nabla^{\varphi}_{e_{j}}\upsilon ,\nabla^{\varphi}_{e_{i}}\tau^{s}(\varphi))\Big\}+ \operatorname{div}_g \eta_{3},
\end{eqnarray*}
where $\displaystyle\eta_{3}(X)=\sum_{i,j=1}^{m}h(d\varphi(X),d\varphi(e_{j}))h(\nabla ^{\varphi}_{e_{j}}\upsilon,\tau^{s}(\varphi))$
\end{lemma}
\begin{proof}[Proof of Lemma \ref{Lem3}]
We compute
\begin{eqnarray*}
  \sum_{i,j=1}^{m}\frac{\partial}{\partial t}[h(d\varphi_{t}(e_{i}),d\varphi_{t}(e_{j}))\nabla d\varphi_{t}(e_{i},e_{j})] &=& \sum_{i,j=1}^{m}\Big\{h(\nabla^{\phi}_{e_{i}}d\phi(\partial t),d\varphi_{t}(e_{j}))\nabla d\varphi_{t}(e_{i},e_{j})\\
   &+&h(d\varphi_{t}(e_{i}),\nabla^{\phi}_{e_{j}}d\phi(\partial t))\nabla^{\phi}_{e_{i}}d\varphi_{t}(e_{j}) \\
   &+& h(d\varphi_{t}(e_{i}),d\varphi_{t}(e_{j}))\nabla^{\phi}_{\partial t}\nabla^{\phi}_{e_{i}}d\varphi_{t}(e_{j})\Big\}.
\end{eqnarray*}
We find that\\

$\displaystyle\sum_{i,j=1}^{m}h\Big(\frac{\partial}{\partial t}[h(d\varphi_{t}(e_{i}),d\varphi_{t}(e_{j}))\nabla d\varphi_{t}(e_{i},e_{j})],\tau^{s}(\varphi_{t})\Big)\Big|_{t=0}$
\begin{eqnarray}\label{E9}
\qquad\qquad\qquad\qquad\nonumber  &=& \sum_{i,j=1}^{m}\Big\{h(\nabla^{\varphi}_{e_{i}}\upsilon,d\varphi(e_{j}))h(\nabla d\varphi(e_{i},e_{j}),\tau^{s}(\varphi))\\
\nonumber &+& h(d\varphi(e_{i}),d\varphi(e_{j}))h(\nabla^{\phi}_{\partial t}\nabla^{\phi}_{e_{i}}d\varphi_{t}(e_{j}),\tau^{s}(\varphi_{t}))\big|_{t=0}\\
&+&h(d\varphi(e_{i}),\nabla^{\varphi}_{e_{j}}\upsilon)h(\nabla d\varphi(e_{i},e_{j}),\tau^{s}(\varphi))\Big\}.
\end{eqnarray}
The second term in the right hand of (\ref{E9}), becomes\\

$\displaystyle\sum_{i,j=1}^{m}h(d\varphi(e_{i}),d\varphi(e_{j}))h(\nabla^{\phi}_{\partial t}\nabla^{\phi}_{e_{i}}d\varphi(e_{j}),\tau^{s}(\varphi_{t}))\Big|_{t=0} $
\begin{eqnarray} \label{E10}
 \qquad\qquad\qquad\qquad\nonumber &=& \sum_{i,j=1}^{m}\Big\{
  h\big(d\varphi(e_{i}),d\varphi(e_{j})\big)h\big(\nabla ^{\varphi}_{e_{i}}\nabla ^{\varphi}_{e_{j}}\upsilon,\tau^{s}(\varphi)\big) \\\nonumber
  &+&h(d\varphi(e_{i}),d\varphi(e_{j}))h(R^{N}(\upsilon, d\varphi(e_{i}))d\varphi(e_{j}),\tau^{s}(\varphi))\Big\}\\
\nonumber  &=&\sum_{i,j=1}^{m}\Big\{e_{i}\Big[h(d\varphi(e_{i}),d\varphi(e_{j}))h(\nabla ^{\varphi}_{e_{j}}\upsilon,\tau^{s}(\varphi))\Big]\\ \nonumber
&+& h(d\varphi(e_{i}),d\varphi(e_{j}))h(R^{N}(\upsilon , d\varphi(e_{i}))d\varphi(e_{j}),\tau^{s}(\varphi))\\
 \nonumber &-&h(\tau(\varphi),d\varphi(e_{j}))h(\nabla ^{\varphi}_{e_{j}}\upsilon ,\tau^{s}(\varphi))\\
 \nonumber &-&h(d\varphi(e_{i}),\nabla^{\varphi}_{e_{i}}d\varphi(e_{j}))h(\nabla ^{\varphi}_{e_{j}}\upsilon ,\tau^{s}(\varphi))\\
&-& h(d\varphi(e_{i}),d\varphi(e_{j}))h(\nabla ^{\varphi}_{e_{j}}\upsilon ,\nabla^{\varphi}_{e_{i}}\tau^{s}(\varphi))\Big\}.
  \end{eqnarray}
By replacing (\ref{E10}) in (\ref{E9}), we get the result of Lemma (\ref{Lem3}).
\end{proof}
\begin{proof}[Proof of Theorem \ref{thm2}]
The equation (\ref{E11}) and three Lemmas, (\ref{lem1}), (\ref{Lem2}), (\ref{Lem3}), gives us
\begin{eqnarray}\label{E12}
 \nonumber
h(\frac{\partial}{\partial t}\tau^{s}(\varphi_{t}),\tau^{s}(\varphi_{t}))\Big|_{t=0}&=&
\sum_{i,j=1}^{m}\Big\{h(d\varphi(e_{i}),d\varphi(e_{j}))h(R^{N}(\upsilon, d\varphi(e_{i}))d\varphi(e_{j}),\tau^{s}(\varphi))\\\nonumber
 &-& h(d\varphi(e_{i}),d\varphi(e_{j}))h(\nabla ^{\varphi}_{e_{j}}\upsilon ,\nabla^{\varphi}_{e_{i}}\tau^{s}(\varphi))\\\nonumber
 &-&h(\nabla^{\varphi}_{e_{i}}d\varphi(\upsilon ,d\varphi(e_{j}))h(d\varphi(e_{j}),\nabla^{\varphi}_{e_{i}}\tau^{s}(\varphi))\\
 &-&h(d\varphi(e_{i}),\nabla ^{\varphi}_{e_{j}}\upsilon)h(d\varphi(e_{j}),\nabla^{\varphi}_{e_{i}}\tau^{s}(\varphi))\Big\}.
\end{eqnarray}
For the second term in the right hand of (\ref{E12}), we have
\begin{eqnarray}\label{E13}
 \nonumber -\sum_{i,j=1}^{m}h\big(d\varphi(e_{i}),d\varphi(e_{j})\big)h\big(\nabla ^{\varphi}_{e_{j}}\upsilon,\nabla^{\varphi}_{e_{i}}\tau^{s}(\varphi)\big)
 &=&-\sum_{i,j=1}^{m}e_{j}\Big[ h\big(d\varphi(e_{i}),d\varphi(e_{j})\big)h\big(\upsilon,\nabla^{\varphi}_{e_{i}}\tau^{s}(\varphi)\big)\Big] \\ \nonumber
 \nonumber  &+& \sum_{i,j=1}^{m}h\big(\nabla ^{\varphi}_{e_{j}}d\varphi(e_{i}),d\varphi(e_{j})\big)h\big(\upsilon,\nabla^{\varphi}_{e_{i}}\tau^{s}(\varphi)\big) \\
 \nonumber  &+& \sum_{i=1}^{m}h\big(d\varphi(e_{i}),\tau(\varphi)\big)h\big(\upsilon,\nabla^{\varphi}_{e_{i}}\tau^{s}(\varphi)\big)\\
   &+&\sum_{i,j=1}^{m}h\big(d\varphi(e_{i}),d\varphi(e_{j})\big)h\big(\upsilon,\nabla ^{\varphi}_{e_{j}}\nabla^{\varphi}_{e_{i}}\tau^{s}(\varphi)\big).\qquad\qquad
\end{eqnarray}
For the third term in the right hand of (\ref{E12}), we have
\begin{eqnarray}\label{E14}
\nonumber -\sum_{i,j=1}^{m}h\big(\nabla^{\varphi}_{e_{i}}\upsilon ,d\varphi(e_{j})\big)h\big(d\varphi(e_{j}),\nabla^{\varphi}_{e_{i}}\tau^{s}(\varphi)\big)
&=&-\sum_{i,j=1}^{m}e_{i}\Big[
  h\big(\upsilon,d\varphi(e_{j})\big)h\big(d\varphi(e_{j}),\nabla^{\varphi}_{e_{i}}\tau^{s}(\varphi)\big)\Big] \\ \nonumber
  &+& \sum_{i,j=1}^{m}h\big(\upsilon,\nabla^{\varphi}_{e_{i}}d\varphi(e_{j})\big)h\big(d\varphi(e_{j}),\nabla^{\varphi}_{e_{i}}\tau^{s}(\varphi)\big)\\  \nonumber
  &+& \sum_{i,j=1}^{m}h\big(\upsilon,d\varphi(e_{j})\big)h\big(\nabla^{\varphi}_{e_{i}}d\varphi(e_{j}),\nabla^{\varphi}_{e_{i}}\tau^{s}(\varphi)\big) \\
  &+& \sum_{i,j=1}^{m} h\big(\upsilon,d\varphi(e_{j})\big)
 h\big(d\varphi(e_{j}),\nabla^{\varphi}_{e_{i}}\nabla^{\varphi}_{e_{i}}\tau^{s}(\varphi)\big).\qquad\qquad
 \end{eqnarray}
  The fourth  term in the right hand of (\ref{E12}) becomes
 \begin{eqnarray}\label{E15}
 \nonumber  -\sum_{i,j=1}^{m}h\big(d\varphi(e_{i}),\nabla ^{\varphi}_{e_{j}}\upsilon \big) h\big(d\varphi(e_{j}),\nabla^{\varphi}_{e_{i}}\tau^{s}(\varphi)\big)
 &=&-\sum_{i,j=1}^{m}e_{j}\Big[
   h\big(d\varphi(e_{i}),\upsilon \big) h\big(d\varphi(e_{j}),\nabla^{\varphi}_{e_{i}}\tau^{s}(\varphi)\big)\Big] \\ \nonumber
    &+& \sum_{i,j=1}^{m}h\big(\nabla ^{\varphi}_{e_{j}}d\varphi(e_{i}),\upsilon \big) h\big(d\varphi(e_{j}),\nabla^{\varphi}_{e_{i}}\tau^{s}(\varphi)\big) \\ \nonumber
 &+&\sum_{i=1}^{m}h\big(d\varphi(e_{i}),\upsilon \big) h\big(\tau(\varphi),\nabla^{\varphi}_{e_{i}}\tau^{s}(\varphi)\big)  \\
 &+&  \sum_{i,j=1}^{m}h\big(d\varphi(e_{i}),\upsilon \big) h\big(d\varphi(e_{j}),\nabla ^{\varphi}_{e_{j}}\nabla^{\varphi}_{e_{i}}\tau^{s}(\varphi)\big).\qquad\qquad
 \end{eqnarray}
 By replacing  (\ref{E15}),  (\ref{E14}) and  (\ref{E13})  in (\ref{E12}) and integrating  over $M$ compact, we deduce the result of Theorem \ref{thm2}.
\end{proof}
\begin{remark}
Let $\varphi$ be a smooth map from compact Riemannian manifold $(M,g)$ to Riemannian manifold $(N,h)$.
We have, $\tau^s_2(\varphi)=J^s_{\varphi}(\tau^s(\varphi))$. Moreover, the map $\varphi$ is bi-symphonic if and only if $\tau^s_2(\varphi)=0$.
\end{remark}
\begin{example}
A smooth curve $\gamma:I\subseteq\mathbb{R}\longrightarrow\mathbb{R}^n$, $t\longmapsto(\gamma_1(t),...,\gamma_n(t))$ is bi-symphonic if and only if
$$14\gamma_{i}^{(1)}(t)^2\gamma_{i}^{(2)}(t)^3+17\gamma_{i}^{(1)}(t)^3\gamma_{i}^{(2)}(t)\gamma_{i}^{(3)}(t)+2\gamma_{i}^{(1)}(t)^4\gamma_{i}^{(4)}(t)=0,\quad \forall i=1,...,n.$$
If $\gamma_i(t)=t^{4/3}$, or $t^{15/11}$ with $I=(0,\infty)$, the curve $\gamma$ is bi-symphonic non-symphonic. Here, $\tau^s(\gamma)=3\sum_{i=1}^n(\gamma_{i}^{(1)})^2\gamma_{i}^{(2)}\partial_i$.
\end{example}
\begin{example}
For all $m\geq2$,  the canonical inclusion $\varphi:\mathbb{S}^m\hookrightarrow\mathbb{R}^{m+1}$ is not bi-symphonic map. Indeed;
let $\{e_i\}_{i=1,...,m}$ be an orthonormal frame with respect to induced Riemannian metric on $\mathbb{S}^m$ by the inner product $<,>$ on $\mathbb{R}^n$,
and $\nabla d\varphi$ the second fundamental form of the sphere $\mathbb{S}^m$ on $\mathbb{R}^{m+1}$, then $(\nabla d\varphi)(X,Y)=-<X,Y>P$, for all $X,Y\in\Gamma(T\mathbb{S}^m)$,
where $P$ is the position vector field on $\mathbb{R}^{m+1}$. One can verify by direct computations that $\tau^s(\varphi)   = -m(P\circ\varphi)$,
$$2\sum_{i,j=1}^{m}h(\nabla^{\varphi}_{e_{i}}\tau^{s}(\varphi),d\varphi(e_{j}))\nabla d\varphi(e_{i},e_{j})   =  2m^2(P\circ\varphi),$$
$$\sum_{i,j=1}^{m}\Big[ h(\nabla^{\varphi}_{e_{i}}\nabla^{\varphi}_{e_{i}}\tau^{s}(\varphi)-
   \nabla^{\varphi}_{\nabla^{M}_{e_{i}}e_{i}}\tau^{s}(\varphi),d\varphi(e_{j}))+
   h(\nabla^{\varphi}_{e_{j}}\tau^{s}(\varphi),\nabla d\varphi(e_{i},e_{i}))\Big]d\varphi(e_{j})=0,$$
$$\sum_{i,j=1}^{m}\Big[ h(\nabla d\varphi(e_{i},e_{j}),d\varphi(e_{j}))+
   h(d\varphi(e_{i}),\nabla d\varphi(e_{j},e_{j}))\Big]\nabla^{\varphi}_{e_{i}}\tau^{s}(\varphi)=0,$$
and the following
\begin{eqnarray*}
\sum_{i,j=1}^{m}  h(d\varphi(e_{i}),d\varphi(e_{j}))\Big[\nabla^{\varphi}_{e_{j}}\nabla^{\varphi}_{e_{i}}\tau^{s}(\varphi)&-&\nabla^{\varphi}_{\nabla^{M}_{e_{j}}e_{i}}\tau^{s}(\varphi)\\
    &+&
   R^{\mathbb{R}^{m+1}}(\tau^{s}(\varphi),d\varphi(e_{j}))d\varphi(e_{i})\Big]=m^2(P\circ\varphi).
\end{eqnarray*}
According to Theorem \ref{thm2} we find that $\tau^s_2(\varphi)   =3m^2(P\circ\varphi)$.
\end{example}

\subsection*{Conflict of interest}
The author states that there is no conflict of interest.
\subsection*{Data availability}
Not applicable.

\end{document}